\newif\ifarxiv
\newif\ifecp
\newif\ifims
\newif\ifincludeleftovers

\arxivtrue

\ifecp
\documentclass[ECP]{ejpecp}
\usepackage{enumitem}
\fi

\ifims
\documentclass[bj,numbers]{imsart}
\fi
\ifarxiv
\documentclass[11pt,a4paper]{article}
\usepackage[numbers]{natbib}
\usepackage[a4paper
]{geometry}
\fi

\ifarxiv
\usepackage{enumitem}
\usepackage[colorlinks=true, allcolors=blue]{hyperref}
\usepackage{datetime}
\newdateformat{mydate}{\THEDAY \ \monthname[\THEMONTH] \THEYEAR}
\mydate
\fi

\usepackage[utf8]{inputenc}
\usepackage{amsmath,amsthm,amssymb}
\usepackage{lslbasicmath}
\usepackage{lslnames}
\usepackage{lslabbrv}
\usepackage{lslhats}
\usepackage{lslcomm}

\ifims
\startlocaldefs
\fi

\numberwithin{equation}{section}

\ifarxiv
\theoremstyle{plain}
\newtheorem{theorem}{Theorem}[section]
\newtheorem{proposition}[theorem]{Proposition}
\newtheorem{lemma}[theorem]{Lemma}
\theoremstyle{definition}
\newtheorem{definition}[theorem]{Definition}
\newtheorem{remark}[theorem]{Remark}
\newtheorem{example}[theorem]{Example}
\newtheorem*{remark*}{Remark}
\fi

\ifims
\theoremstyle{plain}
\newtheorem{theorem}{Theorem}[section]
\newtheorem{proposition}[theorem]{Proposition}
\newtheorem{lemma}[theorem]{Lemma}
\theoremstyle{definition}

\newtheorem{example}[theorem]{Example}
\newtheorem*{remark*}{Remark}
\fi

\newcommand{\sgnorm}[1]{\norm{#1}_{\psi_2}}
\newcommand{\onorm}[1]{\norm{#1}_{\psi}}

\newcommand{\ex}{\operatorname{ex}}

\ifims
\endlocaldefs
\fi

\ifecp

\SHORTTITLE{Sharp constants relating the sub-Gaussian norm and parameter}

\TITLE{Sharp constants relating the sub-Gaussian norm and the sub-Gaussian parameter}

\AUTHORS{%
  Lasse \Leskela\footnote{Aalto University, Finland.
    \EMAIL{lasse.leskela@aalto.fi}}\orcid{0000-0001-8411-8329}
  \and
  Matvei Zhukov\footnote{Aalto University, Finland.
    \EMAIL{matvei.zhukov@aalto.fi}}}%


\AMSSUBJ{60E05; 60E15} 

\SUBMITTED{July 9, 2025} 
\ACCEPTED{TBA, 2025} 

\VOLUME{0}
\YEAR{2025}
\PAPERNUM{0}
\DOI{10.1214/YY-TN}

\ABSTRACT{
We determine the optimal constants in the classical inequalities relating
the sub-Gaussian norm \(\|X\|_{\psi_2}\) and
the sub-Gaussian parameter \(\sigma_X\) for centered real-valued random variables.
We show that \(\sqrt{3/8} \cdot \|X\|_{\psi_2} \le \sigma_X \le \sqrt{\log 2} \cdot \|X\|_{\psi_2}\),
and that both bounds are sharp, attained by the standard Gaussian and Rademacher distributions, respectively.
}

\fi

\begin{document}

\ifims
\begin{frontmatter}
\fi

\ifarxiv
\title{Sharp constants relating the sub-Gaussian norm and the sub-Gaussian parameter}
\fi

\ifarxiv
\date{\today}
\author{Lasse Leskelä \and Matvei Zhukov}
\maketitle
\fi

\ifarxiv
\begin{abstract}
We determine the optimal constants in the classical inequalities relating the sub-Gaussian norm \(\|X\|_{\psi_2}\) and the sub-Gaussian parameter \(\sigma_X\) for centered real-valued random variables. We show that \(\sqrt{3/8} \cdot \|X\|_{\psi_2} \le \sigma_X \le \sqrt{\log 2} \cdot \|X\|_{\psi_2}\), and that both bounds are sharp, attained by the standard Gaussian and Rademacher distributions, respectively.
\end{abstract}
\fi

\ifarxiv
\fi

\ifims
\end{frontmatter}
\fi

\section{Introduction}

A real-valued random variable \( X \) is said to be sub-Gaussian if its tails decay at least as fast as those of a Gaussian random variable. Two standard ways to quantify such behaviour are via the \emph{sub-Gaussian norm}
\[
 \|X\|_{\psi_2} \ = \ \inf\left\{ K > 0 \colon \E e^{X^2 / K^2} \le 2 \right\},
\]
and the \emph{sub-Gaussian parameter}
\[
 \sigma_X \ = \ \inf\left\{ \sigma \ge 0 \colon
 \E e^{s (X - \E X)} \le e^{\sigma^2 s^2 / 2} \ \text{for all } s \in \R \right\}.
\]

It is well known that \(\|X\|_{\psi_2} < \infty\) if and only if \(\E|X| < \infty\) and \(\sigma_X < \infty\), and these equivalent conditions are both commonly used as definitions of a sub-Gaussian random variable \cite{Boucheron_Lugosi_Massart_2013,Vershynin_2018,Wainwright_2019}.
Moreover, there exist universal constants \( c_1, c_2 > 0 \) such that for every centered random variable \( X \),
\[
  c_1 \, \|X\|_{\psi_2} \, \le \, \sigma_X \, \le \, c_2 \, \|X\|_{\psi_2}.
\]
Precise knowledge of these universal constants is important in high-dimensional probability and statistical applications where sharp tail bounds are crucial. However, their exact values do not appear to be explicitly known in the literature. The main goal of this note is to identify them precisely.

\begin{theorem}
\label{the:Main}
For any centered random variable \( X \),
\[
 \sqrt{3/8} \cdot \|X\|_{\psi_2} \le \sigma_X \le \sqrt{\log 2} \cdot \|X\|_{\psi_2},
\]
and both inequalities are sharp.
\end{theorem}

Sharpness is demonstrated by the standard Gaussian and Rademacher distributions, as shown below.

\begin{example}
Let \( X \sim \mathcal{N}(0,1) \). Then \( \sigma_X = 1 \) and \( \|X\|_{\psi_2} = \sqrt{8/3} \), showing that the lower bound in Theorem~\ref{the:Main} is tight.
\end{example}

\begin{example}
Let \( X \in \{-1,1\} \) with \( \pr(X = \pm 1) = 1/2 \). Then \( \sigma_X = 1 \) and \( \|X\|_{\psi_2} = 1 / \sqrt{\log 2} \), showing that the upper bound in Theorem~\ref{the:Main} is tight.
\end{example}

The optimal upper bound \( \sigma_X / \|X\|_{\psi_2} \le \sqrt{\log 2} \approx 0.83 \)
in Theorem~\ref{the:Main} improves upon the best previously known estimate,
to our knowledge, namely \( (3/2)^{1/2} (\log 2)^{1/4} \approx 1.12 \) reported in \cite{Li_2024+}.
Earlier, weaker upper bounds include \( \sqrt{3} \approx 1.73 \) \cite{Rivasplata_2012,VanHandel_2016}, \( 2 \) \cite{Boucheron_Lugosi_Massart_2013,Zhang_Chen_2021}, \( \sqrt{2e} \approx 2.33 \) \cite{Chafai_etal_2012}, and \( \sqrt{8} \approx 2.83 \) \cite{Pollard_2024}.

The optimal lower bound \( \sigma_X / \|X\|_{\psi_2} \ge \sqrt{3/8} \approx 0.61 \) appears implicitly in earlier works \cite{Buldygin_Kozachenko_2000, Li_2024+, Wainwright_2019}, where it can be derived by evaluating specific constants in general sub-Gaussian inequalities. However, to our knowledge, its sharpness has not been previously identified. As demonstrated above, this bound is attained by the standard Gaussian distribution. Other explicit, albeit weaker, lower bounds reported in the literature include \( \sqrt{1/6} \approx 0.41 \) \cite{Pollard_2024, Rivasplata_2012, VanHandel_2016}, \( \sqrt{1/8} \approx 0.35 \) \cite{Boucheron_Lugosi_Massart_2013}, and \( \sqrt{(\log 2)/8} \approx 0.29 \) \cite{Zhang_Chen_2021}.

The remainder of the paper is organised as follows. Section~\ref{sec:UpperBound} contains the main argument, providing a detailed proof of the upper bound in Theorem~\ref{the:Main}. For completeness, Section~\ref{sec:LowerBound} includes a short proof of the lower bound, although it can be deduced from earlier works \cite{Buldygin_Kozachenko_2000, Li_2024+, Wainwright_2019}.

\section{Proof of upper bound}
\label{sec:UpperBound}

In this section we prove the upper bound of Theorem~\ref{the:Main}.
The core of the proof is Section~\ref{sec:ReductionBinary},
where we show that when maximising the moment generating function
among centered random variables with $\sgnorm{X} \le 1$,
it suffices to restrict to binary random variables.
In Section~\ref{sec:UpperBoundForBinaryDistributions} we prove the
upper bound for binary random variables.
Finally, Section~\ref{sec:ConclusionUB} concludes the proof.

\subsection{Reduction to binary distributions}
\label{sec:ReductionBinary}

In this section we prove that for any real number $s$,
when maximising the moment generating function $\E e^{sX}$ with
respect to centered random variables with
sub-Gaussian norm bounded by $\sgnorm{X} \le 1$,
we may restrict to binary random variables.


In the proof, it is more convenient to operate with probability measures
instead of random variables.
Denote by $\cP$ the set of probability measures on the Borel sigma-algebra of $\R$.
This is the set of laws of all real-valued random variables.
We denote by
\begin{equation}
 \label{eq:MomentSet}
 \cH
 \weq \left\{ \mu \in \cP \colon \int x \, \mu(dx) = 0, \ \int e^{x^2} \mu(dx) \le 2 \right\}.
\end{equation}
the laws of centered random variables $X$ with $\sgnorm{X} \le 1$,
and by
\begin{equation}
\label{eq:MomentSetFinite}
 \cH_m
 \weq \bigg\{ \mu \in \cH \colon \mu = \sum_{i=1}^m p_i \de_{x_i}, \
 \text{$x_1,\dots,x_m$ distinct}, \ 
 p_1,\dots,p_m > 0
 \bigg\}.
\end{equation}
the probability measures in $\cH$ having support of size $m$.

\begin{proposition}
\label{the:ReductionToBinary}
For any real number $s \in \R$, there exists a probability measure
$\mu_* \in \cH_2$ such that
\[
 M_{\mu_*}(s)
 \weq
 \sup_{\mu \in \cH} M_\mu(s).
\]
\end{proposition}

The proof of Proposition~\ref{the:ReductionToBinary} proceeds by showing that for every \( s \):
\begin{enumerate}
\item The map \( \mu \mapsto M_\mu(s) \) attains a maximum on \( \mathcal{H} \) (Section~\ref{sec:MaximumIsAttained}).
\item This maximum satisfies \( \max_{\mu \in \mathcal{H}} M_\mu(s) = \max_{\mu \in \mathcal{H}_2 \cup \mathcal{H}_3} M_\mu(s) \) (Section~\ref{sec:ReductionToFinitelySupportedDistributions}).
\item The map \( \mu \mapsto M_\mu(s) \) does not attain a maximum on \( \mathcal{H}_3 \) (Section~\ref{sec:ReductionToBinaryDistributions}).
\end{enumerate}
We then conclude the proof at the end of Section~\ref{sec:ReductionToBinaryDistributions}.

\subsubsection{Maximum is attained}
\label{sec:MaximumIsAttained}

In this section we prove that for any $s \in \R$,
the moment generating function $\E e^{sX}$ attains
a maximum over the set of centered random variables with
sub-Gaussian norm bounded by $\sgnorm{X} \le 1$.
The laws of such random variables are presented by the set $\cH$ defined
in \eqref{eq:MomentSet}.

\begin{proposition}
\label{the:MaximumIsAttained}
For any $s \in \R$, the functional $\mu \mapsto M_\mu(s)$
attains a maximum on $\cH$.
\end{proposition}

The proof of Proposition~\ref{the:MaximumIsAttained} is based on continuity and compactness with respect to convergence in distribution.
We first derive tails bounds (Lemma~\ref{the:UI}) to prove 
continuity (Lemma~\ref{the:Continuity}) and compactness (Lemma~\ref{the:CompactMomentSet}),
and then finish the proof of Proposition~\ref{the:MaximumIsAttained}.

\begin{lemma}
\label{the:UI}
For any number $s \in \R$ and for any probability measure $\mu$
on $\R$ such that $\int e^{x^2} \mu(dx) \le 2$,
\begin{align}
 \label{eq:ExponentialUI}
 \int e^{\abs{sx}} \, 1(\abs{x} \ge K) \, \mu(dx)
 &\wle 2 e^{-K^2/2}, \qquad K \ge 2\abs{s}, \\
 \label{eq:1stMomentUI}
 \int \abs{x} \, 1(\abs{x} \ge K) \, \mu(dx)
 &\wle 2 e^{-K^2/2}, \qquad K \ge 2.
\end{align}
\end{lemma}
\begin{proof}
Fix $K \ge 2 \abs{s}$.
Observe that $x^2 - \abs{sx} \ge \frac12 x^2$ for $\abs{x} \ge 2\abs{s}$.
Hence,
\[
 e^{\abs{sx}}
 \weq \frac{e^{x^2}}{e^{x^2 - \abs{sx}}}
 \wle \frac{e^{x^2}}{e^{x^2/2}}
 \wle \frac{e^{x^2}}{e^{K^2/2}}
 \qquad \text{for $\abs{x} \ge K$}.
\]
Because $\int e^{x^2} \mu(dx) \le 2$, inequality \eqref{eq:ExponentialUI} follows by noting that
\[
 \int e^{\abs{sx}} \, 1(\abs{x} \ge K) \, \mu(dx)
 \wle e^{-K^2/2} \int e^{x^2} \mu(dx)
 \wle 2 e^{-K^2/2}.
\]
Inequality \eqref{eq:1stMomentUI} follows by
noting that $\abs{x} \le e^{\abs{x}}$ and applying 
\eqref{eq:ExponentialUI} with $s=1$.
\end{proof}

\begin{lemma}
\label{the:Continuity}
For any $s \in \R$, $\mu \mapsto M_\mu(s)$
is a continuous real functional on $\cH$ in the topology of weak convergence.
\end{lemma}
\begin{proof}
Assume that $\mu_n \to \mu_\infty$ weakly, for some $\mu_n, \mu_\infty \in \cH$.
By a Skorohod coupling \cite[Theorem 4.30]{Kallenberg_2002}, there exist random variables $X_n, X_\infty$
defined on a common probability space
such that 
$\law(X_n) = \mu_n$,
$\law(X_\infty) = \mu_\infty$,
and
$X_n \to X_\infty$ almost surely.
Then $e^{s X_n} \to e^{s X_\infty}$ almost surely.
The random variables $e^{s X_n}$ are uniformly integrable
due to \eqref{eq:ExponentialUI}.
It follows \cite[Lemma 4.11]{Kallenberg_2002} that
\(
 M_{\mu_n}(s)
 = \E e^{s X_n}
 \to \E e^{s X_\infty}
 = M_{\mu_\infty}(s).
\)
\end{proof}

\begin{lemma}
\label{the:CompactMomentSet}
The set $\cH \subset \cP$ is compact in the topology of weak convergence.
\end{lemma}
\begin{proof}
Let us first verify that $\cH$ is closed.
Assume that $\mu_n \to \mu$ weakly for some $\mu_n \in \cH$ and $\mu_\infty \in \cP$.
By a Skorohod coupling \cite[Theorem 4.30]{Kallenberg_2002},
there exist random variables $X_n, X_\infty$ defined on a common probability space such that 
$\law(X_n) = \mu_n$,
$\law(X_\infty) = \mu_\infty$,
and
$X_n \to X_\infty$ almost surely.
Then by Fatou's lemma \cite[Lemma 4.11]{Kallenberg_2002},
\begin{equation}
 \label{eq:CompactMomentSet1}
 \E e^{X_\infty^2}
 \weq \E \liminf_{n \to \infty} e^{X_n^2}
 \wle \liminf_{n \to \infty} \E e^{X_n^2}
 \wle 2.
\end{equation}
The random variables $X_n$ are uniformly integrable
due to \eqref{eq:1stMomentUI}.
It follows \cite[Lemma 4.11]{Kallenberg_2002} that $\E \abs{X_\infty} < \infty$ and
\begin{equation}
 \label{eq:CompactMomentSet2}
 \E X_\infty 
 \weq \lim_{n \to \infty} \E X_n
 \weq 0.
\end{equation}
Because $\mu_\infty = \law(X_\infty)$,
we see in the light of \eqref{eq:CompactMomentSet1}--\eqref{eq:CompactMomentSet2}
that $\mu_\infty \in \cH$.  We conclude that $\cH$ is closed.

By \eqref{eq:1stMomentUI} we see that $\cH$ is uniformly integrable,
and hence tight \cite{Leskela_Vihola_2013}.
Prokhorov's theorem 
(see e.g.\ \cite[Theorem 3.2.2]{Ethier_Kurtz_1986} or \cite[Theorem 16.3]{Kallenberg_2002})
then implies that $\cH$ is relatively compact.
Being closed and relatively compact, we conclude that $\cH$ is compact.
\end{proof}

\begin{proof}[Proof of Proposition~\ref{the:MaximumIsAttained}]
We know that $\mu \mapsto M_\mu(s)$ is a continuous (Lemma~\ref{the:Continuity})
functional on the compact (Lemma~\ref{the:CompactMomentSet}) set $\cH$.
Therefore, $\mu \mapsto M_\mu(s)$ attains a maximum on $\cH$.
\end{proof}

\subsubsection{Reduction to finitely supported distributions}
\label{sec:ReductionToFinitelySupportedDistributions}


Because the set of probability measures $\cH$ defined in \eqref{eq:MomentSet}
is convex, we expect the linear functional $\mu \mapsto \int e^{sx} \mu(dx)$
to be maximised in the extremal points of $\cH$.
An element $\mu$ in $\cH$ is called \emph{extremal}
if it cannot be written as a convex combination $\mu = (1-\la) \mu_1 + \la \mu_2$
for some $\mu_1,\mu_2 \in \cH$ and $\la \in (0,1)$.
We denote by $\ex(\cH)$ the extremal elements of $\cH$.
We will apply Winkler's theorem \cite{Winkler_1988} to reduce
our optimisation problem to from general probability measures
to probability measures of finite support. 
Note that both the original and the reduced feasible set are infinite-dimensional.

\begin{lemma}
\label{the:ExtremalElements}
The set of extremal elements of $\cH$ equals
\(
 \ex(\cH)
 = \cH_1 \cup \cH_2 \cup \cH_3.
\)
\end{lemma}
\begin{proof}
Using functions $f_1(x) = x$ and $f_2(x) = e^{x^2}$, we see that
\begin{equation}
 \label{eq:MomentSetIneq}
 \cH
 \weq \{ \mu \in \cP \colon \mu(f_1) \le 0, \, \mu(-f_1) \le 0, \, \mu(f_2) \le 2\}.
\end{equation}
By applying \cite[Theorem 2.1:(a), Example 2.1:(a)]{Winkler_1988}
to the moment set \eqref{eq:MomentSetIneq}, 
we see that all extremal elements of $\cH$ can be expressed in the form
\[
 \mu
 \weq \sum_{i=1}^m p_i \de_{x_i},
\]
where $1 \le m \le 4$, $p_1,\dots,p_m > 0$, $p_1+\cdots+p_m=1$, and the vectors
$(1,x_i, -x_i, e^{x_i^2})$, $i=1,\dots,m$, are linearly independent.
By noting that $\sum_{i=1}^m c_i (1,x_i, -x_i, e^{x_i^2}) = 0$
if and only if $\sum_{i=1}^m c_i (1, x_i, e^{x_i^2}) = 0$,
it follows that the last requirement is equivalent to 
stating that the vectors
$(1,x_i, e^{x_i^2})$, $i=1,\dots,m$, are linearly independent.
Because these vectors are elements of $\R^3$, we must have
$1 \le m \le 3$.
In light of Lemma~\ref{the:LinIndGradients}, we also see that
for $m=3$, linear independence is equivalent to $x_1,\dots,x_m$ being distinct.
The same is obviously true for $m=1,2$, so the claim follows.
\end{proof}

\begin{lemma}
\label{the:ReductionToFinite}
For any real number $s \in \R$, there exists a probability
measure
$\mu_* \in \cH_2 \cup \cH_3$
such that
\[
 M_{\mu_*}(s)
 \weq
 \sup_{\mu \in \cH} M_\mu(s).
\]
\end{lemma}

\begin{proof}
Fix $s \in \R$, and
denote $\al = \sup_{\mu \in \cH} M_\mu(s)$.
By Proposition~\ref{the:MaximumIsAttained}, there exists 
a probability measure $\bar\mu \in \cH$ such that $M_{\bar\mu}(s) = \al$. 
By recalling the representation \eqref{eq:MomentSetIneq}
and applying \cite[Theorem 3.1]{Winkler_1988}, we see that every probability measure
in $\cH$, and in particular $\bar\mu$, can be
represented as a convex combination
\[
 \bar\mu(B) \weq \int_{\ex(\cH)} \mu(B) \, \pi(d\mu)
\]
for some probability measure $\pi$ on the set $\ex(\cH)$ equipped with the
sigma-algebra of evaluation maps $\mu \mapsto \mu(B)$, $B \in \cB(\R)$.
It follows that
\[
 \al
 \weq \int_{\R} e^{sx} \bar\mu(dx)
 \weq \int_{\ex(\cH)} \int_{\R} e^{sx} \mu(dx) \, \pi(d\mu)
 \weq \int_{\ex(\cH)} M_{\mu}(s) \, \pi(d\mu).
\]
Therefore,
\[
 \int_{\ex(\cH)} (\al - M_{\mu}(s)) \, \pi(d\mu)
 \weq 0,
\]
and because the integrand is nonnegative, we conclude that
$M_{\mu}(s) = \al$ for $\pi$-almost every $\mu$.
In particular, the set $\{\mu \in \ex(\cH) \colon M_{\mu}(s) = \al\}$
is nonempty.  We conclude that there exists a probability measure $\mu_* \in \ex(\cH)$
such that 
\[
 M_{\mu_*}(s)
 \weq \sup_{\mu \in \cH} M_\mu(s).
\]

By Lemma~\ref{the:ExtremalElements}, we find that
\(
 \ex(\cH) = \cH_1 \cup \cH_2 \cup \cH_3.
\)
The unique element in $\cH_1$ is $\de_0$. By Jensen's inequality 
$M_{\de_0}(s) \le M_\mu(s)$ for all $\mu \in \cH$,
so that $\mu \mapsto M_\mu(s)$ is minimised, not maximised, at $\de_0$.
Therefore, we conclude that $\mu_* \in \cH_2 \cup \cH_3$.
\end{proof}

\subsubsection{From finitely supported to binary distributions}
\label{sec:ReductionToBinaryDistributions}

Lemma~\ref{the:NoMaximumOnH3} also shows that $\cH_3$ is not a closed subset of the compact set $\cH$.

\begin{lemma}
\label{the:NoMaximumOnH3}
For any $s \ne 0$, the functional $\mu \mapsto M_{\mu}(s)$ on $\cH$ does not attain a maximum on $\cH_3$.
\end{lemma}
\begin{proof}
Every probability measure on $\R$ having a support of size 3 can be represented as
\begin{equation}
\label{eq:Dirac3}
 \mu \weq \sum_{i=1}^3 p_i \de_{x_i},
\end{equation}
where $(p,x) \in \R^6$ is such that
$x_1,x_2,x_3$ are distinct, and
$p_1,p_2,p_3 > 0$ satisfy 
$p_1+p_2+p_3 = 1$.
Under such representation, we see that $\mu \in \cH_3$ if and only if
$p_1 x_1 + p_2 x_2 + p_3 x_ 3 = 0$ and
$p_1 e^{x_1^2} + p_2e^{x_2^2} + p_3 e^{x_3^2} \le 2$.
Let us define 
\[
 G_3
 \weq \Big\{ (p,x) \in (0,1)^3 \times \R^3 \colon \text{$x_1,x_2,x_3$ distinct} \Big\},
\]
and
\[
 H_3
 \weq \Big\{ (p,x) \in G_3 \colon \, g_0(p,x) = 0, \, g_1(p,x) = 0, \, g_2(p,x) \le 0 \Big\},
\]
where
$g_k(p,x) = \sum_{i=1}^3 p_i \psi_k(x_i) - c_k$
are defined in terms of
numbers $c_0 = 1$, $c_1=0$, $c_2=2$, and
functions
\(
 \psi_0(t) = 1,
 \psi_1(t) = t,
 \psi_2(t) = e^{t^2}.
\)
For probability measures of form \eqref{eq:Dirac3}, we see that
$\mu \in \cH_3$ if and only if $(p,x) \in H_3$. 
We conclude that
\[
 \sup_{\mu \in \cH_3} M_\mu(s)
 \weq \sup_{(p,x) \in H_3} f(p,x),
\]
where
\[
 f(p,x) \weq \sum_{i=1}^3 p_i e^{s x_i}.
\]
To prove the claim, it suffices to verify that $f \colon G_3 \to \R$ does not attain a maximum on $H_3 \subset G_3$.

We will assume the contrary and derive contradiction. So, let us assume that there exists a point
$(p,x) \in H_3$ such that $f(p,x) \ge f(p',x')$ for all $(p',x') \in H_3$.
The functions $f,g_0,g_1,g_2$ are continuously differentiable on the open set $G_3 \subset \R^6$, 
with gradients given by
\begin{align*}
 \nabla f(p,x) &\weq \big( e^{s x_1}, e^{sx_2}, e^{sx_3}, \,
 p_1 s e^{sx_1}, p_2 s e^{sx_2}, p_3 s e^{sx_3} \big), \\
 \nabla g_0(p,x) &\weq \big( 1, 1, 1, \, 0, 0, 0 \big), \\
 \nabla g_1(p,x) &\weq \big( x_1, x_2, x_3, \, p_1, p_2, p_3 \big), \\
 \nabla g_2(p,x) &\weq \big( e^{x_1^2}, e^{x_2^2}, e^{x_3^2}, \,
 2 p_1 x_1 e^{x_1^2}, 2 p_2 x_2 e^{x_2^2}, 2 p_3 x_3 e^{x_3^2} \big).
\end{align*}
We see that $\nabla g_0(p,x)$, $\nabla g_1(p,x)$, $\nabla g_2(p,x)$ are linearly independent,
because the vectors $(1,1,1)$, $(x_1,x_2,x_3)$, $(e^{x_1^2},  e^{x_2^2},  e^{x_3^2} )$
are linearly independent (Lemma~\ref{the:LinIndGradients}).
By the Lagrange multiplier theorem \cite{McShane_1973}, there exist
real numbers $\la_0, \la_1, \la_2$ such that
\begin{equation}
 \label{eq:LagrangeMultiplier}
 \nabla f(p,x)
 + \la_0 \nabla g_0(p,x) + \la_1 \nabla g_1(p,x) + \la_2 \nabla g_2(p,x)
 \weq 0.
\end{equation}

The gradients can also be written as
\begin{align*}
 \nabla f(p,x)
 &\weq \big( \phi(x_1), \phi(x_2), \phi(x_3), \, p_1 \phi'(x_1), p_2 \phi'(x_2), p_3 \phi'(x_3) \big), \\
 \nabla g_k(p,x)
 &\weq \big( \psi_k(x_1), \psi_k(x_2), \psi_k(x_3), \, p_1 \psi_k'(x_1), p_2 \psi_k'(x_2), p_3 \psi_k'(x_3) \big).
\end{align*}
Because $p_1, p_2, p_3$ are nonzero, 
we see that \eqref{eq:LagrangeMultiplier} is equivalent to requiring that
\begin{align*}
 \phi(x_i) + \sum_{k=0}^2 \la_k \psi_k(x_i) &\weq 0, \\
 \phi'(x_i) + \sum_{k=0}^2 \la_k \psi_k'(x_i) &\weq 0,
\end{align*}
for all $i=1,2,3$. It follows that 
$u(x_i) = 0$ and $u'(x_i) = 0$ for $i=1,2,3$,
where
\[
 u(t)
 \weq \phi(t) + \sum_{k=0}^2 \la_k \psi_k(t).
\]
Because $u$ vanishes at distinct $x_1,x_2,x_3$, Rolle's theorem \cite[Theorem 5.10]{Rudin_1976}
implies that
$u'$ vanishes at some point in each open interval between two consecutive points from $x_1,x_2,x_3$.
We also know that $u'$ vanishes at each of the points $x_1,x_2,x_3$.
Therefore, $u'$ vanishes in at least 5 distinct points.
It follows by Rolle's theorem that $u''$ vanishes 
at some point in each of the 4 open intervals between the zeros of $u'$.  
Therefore, $u''$ has at least 4 zeros.

To arrive at a contradiction, we will show that
\[
 u''(t)
 \weq s^2 e^{st} + 2 \la_2 (1+2t^2) e^{t^2}.
\]
has at most two zeros. If $\la_2 \ge 0$, then $u'' > 0$ does not vanish anywhere.
Assume next that $\la_2 < 0$. Then $u''(t) = 0$ is equivalent to
\[
 s^2 e^{st}
 \weq 2 \abs{\la_2} (1+2t^2) e^{t^2}.
\]
This is equivalent to $w(t) = 0$ with
\[
 w(t)
 \weq t^2-st + \log (1+2t^2) + \log \frac{2 \abs{\la_2}}{s^2}.
\]
A simple computation shows that
\[
 w''(t)
 \weq \frac{6+8t^4}{(1+2t^2)^2}
\]
We see that $w''$ is strictly positive everywhere. Hence $w$ is strictly convex
and has at most two zeros.  Hence $u''$ has at most two zeros.
This contradicts our derivation that $u''$ has at least 4 zeros. 
Therefore, the assumption of the existence of a point $(p,x) \in H_3$ that
maximises $f$ on $H_3$ must be false.
\end{proof}

\begin{proof}[Proof of Proposition~\ref{the:ReductionToBinary}]
To avoid trivialities, we assume that $s \ne 0$.
By Lemma~\ref{the:ReductionToFinite},
we know that $\sup_{\mu \in \cH} M_\mu(s) = M_{\mu_*}(s)$ for some
probability measure $\mu_*$ in $\cH_ 2 \cup \cH_3$.
Lemma~\ref{the:NoMaximumOnH3} implies that $\mu_* \notin \cH_3$. Hence $\mu_* \in \cH_2$.
\end{proof}

\subsection{Upper bound for binary distributions}
\label{sec:UpperBoundForBinaryDistributions}

Recall the definitions of $\cH$ and $\cH_2$ in \eqref{eq:MomentSet}--\eqref{eq:MomentSetFinite}.
We denote the sub-Gaussian parameter of a probability measure $\mu \in \cH$ by
\[
 \sig_\mu
 \weq \inf\left\{\sig \in [0,\infty):
 \int e^{s x} \mu(dx) \le e^{\sig^2 s^2/2} \ \text{for all $s \in \R$}
 \right\}.
\]

\begin{proposition}
\label{the:UpperBoundBinary}
$\sig_\mu \le \sqrt{\log 2}$ for all $\mu \in \cH_2$.
\end{proposition}

\begin{proof}
Fix a $\mu \in \cH_2$.
Being a probability measure supported on two points, we find that
$\mu = p \de_{x_1} + (1-p) \de_{x_2}$ with $p \in (0,1)$ and $x_1 \ne x_2$.
By swapping $x_1$ and $x_2$ if necessary, we may assume without loss of generality that
$p \in [1/2,1)$.
Because $\mu$ is centered, we find that $x_2 = - \frac{p}{1-p} x_1$.
For convenience, let us parametrise $p = \frac{u}{1+u}$ in terms of $u \in [1,\infty)$.
Then $\int e^{x^2} \mu(dx) = G(u,x_1^2)$ where
\[
 G(u,t)
 \weq \frac{u}{1+u} e^t + \frac{1}{1+u} e^{u^2 t}.
\]
In this case we know \cite[Theorem 3.1]{Buldygin_Moskvichova_2013}
that
\[
 \sig_\mu^2
 \weq
 \begin{cases}
  x_1^2 \frac{u^2-1}{2 \log u}, &\qquad u \ne 1, \\
  x_1^2, &\qquad u = 1.
 \end{cases}
\]

Assume first that $u=1$. Then
$\int e^{x^2} \mu(dx) = e^{x_1^2}$,
so that the condition $\int e^{x^2} \mu(dx) \le 2$
implies that
$\sig_\mu^2 = x_1^2 \le \log 2$.

Assume next that $u>1$.
Denote $\al(u) = 2 \log 2 \frac{\log u}{u^2-1}$.
Then our claim amounts to verifying that
$x_1^2 \le \al(u)$ for all $x_1 \in \R$ 
such that $G(u,x_1^2) \le 2$.
Because $G$ is strictly increasing in its second input,
this is equivalent to verifying that $F(u) = G(u,\al(u))$ satisfies $F(u) \ge 2$.
We note that $F$ is continuously differentiable on $(1,\infty)$, with $F(1+) = 2$.
Therefore, it is sufficient to show that $F'(u) \ge 0$ for all $u > 1$.

Observe that
\begin{equation}
 \label{eq:FPrimePre}
 F'(u)
 \weq \partial_1 G(u,\al(u)) + \al'(u) \partial_2 G(u,\al(u)),
\end{equation}
where
\begin{align*}
 \partial_1 G(u, t)
 &\weq (1+u)^{-2} \Big( e^t  + (2u(1+u)t-1) e^{u^2 t} \Big) \\
 \partial_2 G(u, t)
 &\weq u (1+u)^{-1} \Big( e^t + u e^{u^2 t} \Big), \\
 \al'(u)
 &\weq 2 \log 2 \, u^{-1} (u^2-1)^{-2} \Big( u^2 - 1 - 2 u^2  \log u  \Big).
\end{align*}
By substituting the above equations into \eqref{eq:FPrimePre}, and simplifying, we find that
\begin{equation}
 \label{eq:FPrime}
 F'(u) \weq \frac{h_1(u) e^{\al(u)}}{(1+u)(u^2 - 1)^2} + \frac{h_2(u) e^{u^2 \al(u)}}{(1+u)(u^2 - 1)^2},
\end{equation}
where
\begin{align*}
 h_1(u)
 &\weq 2\log 2 \cdot (u^2 - 1 - 2 u^2 \log u) + (1+u)(u-1)^2, \\
 h_2(u)
 &\weq 2\log 2 \cdot (u^3 - u - 2 u \log u) - (1+u)(u-1)^2.
\end{align*}
Differentiation shows that $h_2(1)=h_2'(1)=h_2''(1)=0$, and that
$h_2'''(u) = 2\log 2 \cdot (6 + 2/u^2) - 6 \ge 0$
for $u \ge 1$.  Therefore, $h_2(u) \ge 0$.
Because $e^{u^2 \al(u)} \ge e^{\al(u)}$,
we find by \eqref{eq:FPrime} that
\begin{equation}
 \label{eq:FPrimeLB}
 F'(u)
 \wge \frac{h(u) e^{\al(u)}}{(1+u)(u^2 - 1)^2},
\end{equation}
where $h(u) = h_1(u)+h_2(u)$.
Differentiation shows that
\[
 h(u)
 \weq 2\log 2 \cdot \Big( u^3 + u^2 - u - 1 - 2 u^2 \log u - 2 u \log u \Big) 
\]
satisfies $h(1)=h'(1)=h''(1)=0$,
together with $h'''(u) = 2\log 2 \cdot \big( 6 - \frac{2}{u^2} - \frac{4}{u} \big) \ge 0$.
Therefore, $h(u) \ge 0$ for $u > 1$.  It follows that $F'(u) \ge 0$ for all $u > 1$ due to \eqref{eq:FPrimeLB}.
\end{proof}

\subsection{Concluding the proof of the upper bound}
\label{sec:ConclusionUB}

\begin{proof}[Proof of Theorem~\ref{the:Main}: Upper bound]
Let $X$ be a random variable such that $\E X = 0$ and $\sgnorm{X} < \infty$.
Choose a number $K > \sgnorm{X}$, and define $Y = X/K$.
Then $Y$ is a centered random variable such that
\[
 \E e^{Y^2}
 \weq \E e^{X^2/K^2}
 \wle 2.
\]
We conclude that the probability distribution $\mu$ of $Y$ belongs to $\cH$.

Fix $s \in \R$, and denote by $M_\mu(s) = \E e^{sY}$ the moment generating function of $Y$.
By Proposition~\ref{the:ReductionToBinary}, we know that
for some $\mu_* \in \cH_2$,
\[
 M_{\mu_*}(s)
 \wge
 M_\mu(s).
\]
By Proposition~\ref{the:UpperBoundBinary},
the sub-Gaussian parameter of $\mu_*$ is bounded by
$\sig_{\mu_*} \le \sqrt{\log 2}$,
so that $M_{\mu_*}(s) \le e^{\log 2 \cdot s^2/2}$.
We conclude that the moment generating function of $Y$
is bounded by 
\begin{equation}
 \label{eq:MGFBoundY}
 \E e^{sY}
 \wle e^{\log 2 \cdot s^2/2}
 \qquad \text{for all $s \in \R$}.
\end{equation}

By applying \eqref{eq:MGFBoundY} with input $sK$, 
it follows that for all $s \in \R$,
\[
 \E e^{sX}
 \weq \E e^{(sK)Y}
 \wle e^{\log 2 \cdot (sK)^2/2}
 \weq e^{(\sqrt{\log 2} \, K)^2 s^2/2}
\]
Because the above bound holds for all $K > \sgnorm{X}$, we conclude that
$\sig_X \le \sqrt{\log 2} \, \sgnorm{X}$.
\end{proof}

\section{Proof of lower bound}
\label{sec:LowerBound}

For completeness, we include a short proof of the lower bound in Theorem~\ref{the:Main}, even though the result is implicit in earlier works \cite{Buldygin_Kozachenko_2000, Li_2024+, Wainwright_2019} and can be obtained, e.g., by evaluating \cite[Lemma~1.6]{Buldygin_Kozachenko_2000} or by tracking constants in the proof of \cite[Theorem~2.6]{Wainwright_2019}.

\begin{proof}[Proof of Theorem~\ref{the:Main}: Lower bound]
Fix a number $K > \sqrt{2} \sig_X$,
and recall that
the moment generating function
of the standard Gaussian distribution with
the probability density function
$\phi(x) = \frac{1}{\sqrt{2\pi}} e^{-x^2/2}$
is given by
\(
 e^{t^2/2}
 = \int_{-\infty}^\infty e^{t u} \phi(u) du.
\)
By substituting $t = \sqrt{2} X/K$, taking expectations, and
applying Fubini's theorem,
we find that
\begin{align*}
 \E e^{X^2/K^2}
 \weq \int_{-\infty}^\infty \E e^{(u \sqrt{2}/K) X} \phi(u) \, du.
\end{align*}
The definition of the sub-Gaussian parameter implies that
\(
 \E e^{(u \sqrt{2}/K) X}
 \le e^{\sig_X^2 u^2 / K^2}.
\)
Observe also that
\(
 e^{\sig_X^2 u^2 / K^2} \phi(u)
 = \phi(u/\tau)
\)
where $\tau = (1 - 2 \sig_X^2/K^2)^{-1/2}$.
It follows that
\(
 \int_{-\infty}^\infty \E e^{(u \sqrt{2}/K) X} \phi(u) \, du
 \le \int_{-\infty}^\infty \phi(u/\tau) \, du
 = \tau,
\)
so that
\[
 \E e^{X^2/K^2}
 \wle (1 - 2 \sig_X^2/K^2)^{-1/2}.
\]
This above inequality reveals that
$\E e^{X^2/K^2} \le 2$
whenever $K > \sqrt{8/3} \, \sig_X$.
Therefore,
$\sgnorm{X} \le \sqrt{8/3} \, \sig_X$,
or equivalently, $\sig_X \ge \sqrt{3/8} \, \sgnorm{X}$.
\end{proof}

\appendix

\section{Auxiliary results}

\begin{lemma}
\label{the:LinIndGradients}
For any $x_1,x_2,x_3 \in \R$, the matrix
\[
 V
 =
 \begin{bmatrix}
 1 & 1 & 1 \\
 x_1 & x_2 & x_3 \\
 e^{x_1^2} & e^{x_2^2} & e^{x_3^2} \\
 \end{bmatrix}
\]
is invertible if and only if $x_1,x_2,x_3$ are distinct.
\end{lemma}

\begin{proof}
If $x_1,x_2,x_3$ are not distinct, then neither are the
columns of $V$, and hence $V$ cannot be invertible.

For the converse implication, it suffices to verify that
the rows $v_1,v_2,v_3$ of $V$ are linearly independent
when $x_1,x_2,x_3$ are distinct.  To do this,
assume that
\(
c_1 v_1 + c_2 v_2 + c_3 v_3 = 0
\)
for some scalars $c_1,c_2,c_3$.
Then the function
\[
 f(t) = c_1 + c_2 t + c_3 e^{t^2}
\]
vanishes at three distinct points \( x_1, x_2, x_3 \).
We note that
\(
 f''(t)
 \weq 2c_3(1 + 2t^2) e^{t^2}.
\)
If \( c_3 > 0 \), then \( f \) is strictly convex, and a strictly convex function cannot have three distinct zeros.  
If \( c_3 < 0 \), then \( -f \) is strictly convex, and similarly, \( f \) cannot have three distinct zeros.  
Therefore, \( c_3 = 0 \), which implies that
\[
f(t) = c_1 + c_2 t.
\]

Because \( f \) vanishes at points \( x_1 \ne x_2 \), and
\(
 0 = f(x_2) - f(x_1) = c_2(x_2 - x_1),
\)
we see that \( c_2 = 0 \). 
It follows that \( c_1 = f(x_1) = 0 \).  
Thus, \( c_1 = c_2 = c_3 = 0 \), which confirms that the rows 
\( v_1, v_2, v_3 \) are linearly independent.
\end{proof}

\bibliographystyle{amsplain}
\bibliography{lslReferences}


\end{document}